\documentclass[11pt]{article}
\usepackage{amsmath,color,amsfonts,amsthm}
\usepackage{amssymb,enumerate,verbatim}
\usepackage{hyperref}
\usepackage{tikz}
\usetikzlibrary{calc}

\usepackage{a4wide}

\makeatletter
\newcommand{\globalcolor}[1]{%
  \color{#1}\global\let\default@color\current@color
}
\makeatother

\newtheorem{lemma}{Lemma}[section]

\newtheorem{claim}[lemma]{Claim}
\newtheorem{theorem}[lemma]{Theorem}
\newtheorem{prop}[lemma]{Proposition}
\newtheorem{conj}[lemma]{Conjecture}
\newtheorem{ques}[lemma]{Question}

\theoremstyle{definition}
\newtheorem{definition}[lemma]{Definition}

\newcommand{\card}[1]{\left| #1 \right|}

\date{}

\title{\vspace{-0.9cm}Isomorphic Bisections of Cubic Graphs}
\author{S. Das\thanks{Institut f\"ur Mathematik, Freie Universit\"at Berlin, 14195 Berlin, Germany. E-mail: \textsf{shagnik@mi.fu-berlin.de}. Research supported by the Deutsche Forschungsgemeinschaft (DFG) project 415310276.}
\and A. Pokrovskiy\thanks{Department of Mathematics, University College London. E-mail: \textsf{dr.alexey.pokrovskiy@gmail.com}.}
\and B. Sudakov\thanks{Department of Mathematics, ETH, 8092 Zurich, Switzerland. E-mail: \textsf{benjamin.sudakov@math.ethz.ch}. Research supported in part by SNSF grant 200021\_196965.}
}

\begin{document}
\maketitle

\begin{abstract}
Graph partitioning, or the dividing of a graph into two or more parts based on certain conditions, arises naturally throughout discrete mathematics, and problems of this kind have been studied extensively. In the 1990s, Ando conjectured that the vertices of every cubic graph can be partitioned into two parts that induce isomorphic subgraphs.
Using probabilistic methods together with delicate recolouring arguments, we prove Ando's conjecture for large connected graphs.

\

\end{abstract}

\section{Introduction} \label{sec:intro}

Graph theory enjoys applications to a wide range of disciplines because graphs are incredibly flexible mathematical structures, capable of modelling very complex systems. The study of complicated networks motivates the following important question: to what extent can a graph be decomposed into simpler subgraphs?

Spearheading this line of research is the classic problem of graph colouring, one of the oldest branches of graph theory. When colouring a graph, one seeks to partition the vertices into as few independent (edgeless) sets as possible. Indeed, subgraphs cannot get much simpler than independent sets, but such ambitious goals come at a cost. Not only is the determination of a graph's chromatic number a notoriously difficult problem, but even when dealing with sparse graph classes, one can require many colours. For instance, many $d$-regular graphs cannot be partitioned into much fewer than $d$ independent sets.

Often one does not want to have such a large number of parts, and so it is natural to ask what can be achieved with fewer colours. An early result along these lines was provided by Lov\'asz~\cite{Lov66}, who proved that, given a graph $G$ of maximum degree $d$, some number of colours $t$, and a sequence $d_1, d_2, \hdots, d_t$ with $\sum_i d_i = d - t + 1$, one can $t$-colour the vertices of $G$ such that the $i$th colour class induces a subgraph with maximum degree $d_i$. In particular, we can partition any graph into two subgraphs, each with half the maximum degree.

While reducing the maximum degree certainly simplifies graphs, it still allows for large connected subgraphs within the colour classes. A different objective, therefore, is to find clustered colourings, which are colourings where each monochromatic component is of bounded size. Alon, Ding, Oporowski and Vertigan~\cite{ADOV03} proved, among other results, that any graph with maximum degree four can be two-coloured such that the largest monochromatic components are of order at most $57$. However, they also constructed six-regular graphs with arbitrarily large monochromatic components in every two-colouring. Answering one of their questions, Haxell, Szab\'o and Tardos~\cite{HST03} proved that clustered two-colourings of graphs of maximum degree five always exist. Aside from improving the bounds on the monochromatic component sizes, subsequent research has sought to explore which graph classes admit clustered colourings, and what is possible with more colours. A more general setting, in which the size of monochromatic components can grow with the size of graph, has also been studied. For example, Linial, Matou\v{s}ek, Sheffet and Tardos~\cite{LMST} proved that any planar $n$-vertex graph has a two-colouring in which all monochromatic components have size at most $O(n^{2/3})$, and this is tight. For more results on vertex colourings with small 
monochromatic components, we refer the interested reader to the survey of Wood~\cite{Woo18}.

The above results show that when we restrict our attention to graphs of small maximum degree, there is less room for complexity, and therefore we can prove strong partitioning results. One might further try to achieve more than simply bounding the size of monochromatic components. Indeed, it appears that in the case of cubic graphs, it is also possible to control the structure of such components. An early result along these lines is due to Akiyama, Exoo and Harary~\cite{AEH80}, who proved that every cubic graph admits a two-\emph{edge}-colouring in which every monochromatic component is a path, with a short proof of this result provided by Akiyama and Chv\'atal~\cite{AC81} soon after. Although these results allowed for paths of unbounded length, Bermond, Fouquet, Habib and P\'eroche~\cite{BFHP84} conjectured that paths of length at most five suffice, a result that would be best possible. Partial results, with larger but finite bounds on the path lengths, were obtained by Jackson and Wormald~\cite{JW96} and by Aldred and Wormald~\cite{AW98}, before the conjecture was proven by Thomassen~\cite{Tho99} in 1999.

\begin{theorem}\label{thm:thomassen}
The edges of any cubic graph can be two-coloured such that each monochromatic component is a path of length at most five.
\end{theorem}

One can also try to find different structure in the partition than just finding path forests with small components. One attractive conjecture of Wormald~\cite{Wor87} from 1987 asks whether we can colour the edges of every cubic graph with an even number of edges so that the red and blue subgraphs form \emph{isomorphic} linear forests. This is known to hold for particular classes of graphs --- it was proved for Jaeger graphs in the work of Bermond, Fouquet, Habib and P\'eroche~\cite{BFHP84} and Wormald~\cite{Wor87}, and for some further classes of cubic graphs by Fouquet, Thuillier, Vanherpe and Wojda~\cite{FTVW09}. 
Given Wormald's conjecture, it is then natural to ask for analogues when colouring vertices rather than edges. In this direction, the following striking conjecture was made by Ando in the 1990s.
\begin{conj}\label{conj:ando}
The vertices of any cubic graph can be two-coloured such that the two colour classes induce isomorphic subgraphs.
\end{conj}

This conjecture was first mentioned in print by Abreu, Goedgebeur, Labbate and Mazzuoccolo~\cite{AGLM19}, who drew connections between Conjecture~\ref{conj:ando}, Wormald's conjecture and a conjecture of Ban and Linial~\cite{BL16} (see Section~\ref{sec:conclusion} for discussion of the latter). They further obtained computational results verifying the conjecture for graphs on at most $32$ vertices.
Hence any counterexample to Conjecture~\ref{conj:ando} must have at least $34$ vertices. Observe that any minimal counterexample must be connected, as the components of a graph can be coloured independently. In this paper we essentially resolve this problem and prove that large connected cubic graphs satisfy Ando's conjecture. This shows that there can be at most finitely many minimal counterexamples.

\begin{theorem} \label{thm:main}
Every sufficiently large connected cubic graph admits a two-vertex-colouring $\varphi$ whose colour classes induce isomorphic subgraphs.
\end{theorem}

The remainder of this paper is organized as follows. In Section~\ref{sec:proof}, we present the proof of Theorem~\ref{thm:main}. The proof uses probabilistic methods together with delicate recolouring arguments. In the process of recolouring
we use some gadgets which we call $P_t$-reducers. The existence of these structures is shown in Section~\ref{sec:reducers}. Finally, we make some concluding remarks and discuss open problems in Section~\ref{sec:conclusion}.

\vspace{0.35cm}
\noindent
{\bf Notation.}\, 
We shall call the colours used in our two-colourings red and blue. To show that the red and blue subgraphs are isomorphic, we will need to keep track of the monochromatic components. Given a fixed graph $H$ and a red-/blue-colouring $\varphi$ of the vertices of a graph $G$, we denote by $r_H(G, \varphi)$ the number of red components of $G$ under $\varphi$ that are isomorphic to $H$, and define $b_H(G,\varphi)$ similarly for the blue components. Paths will play a significant role in our argument, and we write $P_t$ for the path of length $t-1$ on $t$ vertices. Given a colouring of a graph, we call the colouring with colours reversed the \emph{opposite} colouring. We will also often explore the neighbourhood of a vertex or a subset in a graph, and write $B_d(v)$ for the radius-$d$ ball around a vertex $v$ and write $B_d(X)$ for 
the set of all vertices of $G$ within distance at most $d$ from a subset $X$. We use $N^d(X)$ for the set of vertices at distance exactly $d$ from $X$, and abbreviate $N^1(X)$ as $N(X)$ (so $N(X)$ is the set of external neighbours of $X$ in $G$). Finally, all logarithms are to the base $e$.

\section{Proving the theorem} \label{sec:proof}

Given a large connected cubic graph, our goal is to colour the vertices such that the colour classes induce isomorphic subgraphs, and we shall find this colouring in two stages. In the first, we take a semi-random vertex colouring, and show that this is very close to having the desired properties. Then, in the second stage, we make some deterministic local recolourings to balance the two subgraphs and ensure they are truly isomorphic.

\subsection{A random colouring} \label{sec:random}

We begin our search for the isomorphic bisection by showing that one may leverage Thomassen's result to define a random colouring that will produce monochromatic subgraphs that are nearly isomorphic. The following proposition collects the properties of this initial colouring.

\begin{prop} \label{prop:random}
For any $d \in \mathbb{N}$ and any sufficiently large cubic graph $G$, there is a red-/blue-colouring $\varphi$ of the vertices for which the following hold.
\begin{itemize}
	\item[(a)] Each monochromatic component is a path of length at most five.
	\item[(b)] For each $2 \le t \le 6$, $\card{r_{P_t}(G,\varphi) - b_{P_t}(G,\varphi) } \le 3 \sqrt{n \log n}$.
	\item[(c)] There are sequences of vertices $(u_i)_{i \in [s]}$ and $(w_i)_{i \in [s]}$, for some $s \ge 2^{-2d-5}n$, such that all the balls $B_d(u_i)$ and $B_d(w_i)$ are pairwise disjoint, and for each $i \in [s]$, $B_d(u_i)$ and $B_d(w_i)$ induce isomorphic subgraphs with opposite colourings.
	\item[(d)] The colouring $\varphi$ is a bisection; that is, there are an equal number of red and blue vertices.
\end{itemize}
\end{prop}

In our proof, we will need to show that several random variables, each a function of the random colouring, lie close to their expected values. For this we employ McDiarmid's Inequality~\cite{McD89}, which bounds large deviations in random variables defined on product probability spaces, provided they do not vary too much in response to changes in individual coordinates.

\begin{theorem} \label{thm:mcdiarmid}
Let $X = (X_1, X_2, \hdots, X_n)$ be a family of independent random variables, with $X_k$ taking values in a set $A_k$ for each $k \in [n]$. Suppose further that there is some real-valued function $f$ defined on $\prod_{k \in [n]} A_k$ and some $c > 0$ such that $\card{f(x)-f(x')} \le c$ whenever $x$ and $x'$ differ only in a single coordinate. Then, for any $m \ge 0$,
\[ \mathbb{P} \left( \card{f(X) - \mathbb{E}[f(X)]} \ge m \right) \le 2 \exp \left( \frac{-2m^2}{c^2n} \right). \]
\end{theorem}

Equipped with this tool, we can now proceed with our proof.

\begin{proof}[Proof of Proposition~\ref{prop:random}]
To prove the proposition, we shall first define a random colouring $\varphi'$ that satisfies properties (a), (b) and (c) while being very close to a bisection. We shall then make a few small changes to obtain a bona fide bisection $\varphi$, maintaining the other properties in the process. 

To define $\varphi'$, we first apply Theorem~\ref{thm:thomassen} to our cubic graph $G$. This results in a partition of the edges of $G$ into two spanning linear forests, $F_1$ and $F_2$, neither of which contains a path of length six or more. We then take $\varphi'$ to be a uniformly random proper colouring of $F_1$. Observe that this is equivalent to selecting, independently for each path component of $F_1$, one of its two proper colourings uniformly. Thus, our probability space is a product space, with each coordinate corresponding to a colouring of a path component of $F_1$, of which there are at most $n$.

\medskip

We now verify in turn that the first three properties hold. For (a), observe that since $\varphi'$ is a proper colouring of $F_1$, the only monochromatic components are subgraphs of $F_2$. It thus follows immediately that each such component is a path of length at most five.

\medskip

To establish (b), we seek to apply Theorem~\ref{thm:mcdiarmid} to the random variable $r_{P_t}(G,\varphi')$. To do so, we must understand how changing the colouring of an individual path component of $F_1$ can affect $r_{P_t}(G, \varphi')$. Since the component in $F_2$ of each vertex is simply a path, changing the colouring of an individual vertex can affect at most two monochromatic copies of $P_t$. As each path component in $F_1$ is of length at most five, the colouring of each such path therefore affects at most $12$ copies of $P_t$. Hence, we can take $c = 12$ in Theorem~\ref{thm:mcdiarmid}. It then follows that, for each $m > 0$, $\mathbb{P}\left( \card{ r_{P_t}(G,\varphi') - \mathbb{E}[ r_{P_t}(G,\varphi') ] } \ge m \right) \le 2 \exp \left( - m^2 / (72n) \right)$. In particular, the probability that $r_{P_t}(G, \varphi')$ deviates from its expectation by more than $\sqrt{n \log n}$ is at most $2n^{-1/72} = o(1)$.

The same argument holds for $b_{P_t}(G,\varphi')$, and, by symmetry, we have $\mathbb{E}[r_{P_t}(G,\varphi')] = \mathbb{E}[b_{P_t}(G,\varphi')]$. Hence, the probability that $\card{r_{P_t}(G,\varphi') - b_{P_t}(G,\varphi')} \ge 2 \sqrt{n \log n}$ is $o(1)$. Taking a union bound over all choices of $2 \le t \le 6$, we find that with high probability we have $\card{r_{P_t}(G, \varphi') - b_{P_t}(G, \varphi')} \le 2 \sqrt{n \log n}$ for all $2 \le t \le 6$.

\medskip

We next turn our attention to (c). To start, we wish to select a sequence of vertices whose pairwise distances are all at least $2d+1$. Observe that, since $G$ is cubic, for each vertex $v$ we have $\card{B_{2d}(v)} \le 1 + \sum_{i=1}^{2d} 3 \cdot 2^i < 3 \cdot 2^{2d+1}$. We can then greedily select our desired vertices. Indeed, after we add a vertex $v_i$ to our sequence, we eliminate the vertices in $B_{2d}(v_i)$ from our consideration. This guarantees that we are now free to select any of the remaining vertices. This results in a sequence of vertices $v_1, v_2, \hdots, v_{s'}$ for some $s' \ge \tfrac13 2^{-2d-1} n$, with each pair at distance $2d+1$ or more. In particular, the balls $B_d(v_i)$ are pairwise disjoint.

Now consider the subgraphs induced by each ball $B_d(v_i)$. These are subcubic graphs on at most $3 \cdot 2^{d+1}$ vertices, and hence there are a finite number of possible isomorphism types. Furthermore, the subgraphs inherit a vertex-colouring from $\varphi'$. Since the number of vertices in the ball is bounded, there can be at most $2^{3 \cdot 2^{d+1}}$ different colourings. Therefore, considering the isomorphism type and colouring of each ball, there is some finite number $\kappa = \kappa(d)$ of different classes the balls can fall into.

For each $j \in [\kappa]$, we denote by $Y_j$ the random variable counting the number of balls $B_d(v_i)$, $i \in [s']$, that belong to the $j$th class. Note that for each $i$, the distribution of the colouring of $B_d(v_i)$ depends on how the paths of $F_1$ appear in the ball. However, by symmetry, opposite colourings of $B_d(v_i)$ appear with equal probabilities. It follows that if $j$ and $\bar{j}$ represent opposite colourings of isomorphic balls, then $\mathbb{E}[Y_j] = \mathbb{E}[Y_{\bar{j}}]$.

Since the balls are pairwise disjoint, changing the colour of a single vertex can affect at most one ball $B_d(v_i)$. Hence, changing the colouring of a path component of $F_1$ can affect at most six balls. Applying Theorem~\ref{thm:mcdiarmid} to the random variables $Y_j$, we can therefore take $c = 6$. Setting $m = \sqrt{n \log n}$, we find that with probability $1 - o(1)$ we have $\card{Y_j - \mathbb{E}[Y_j]} \le \sqrt{n \log n}$ for each $j \in [\kappa]$. It then follows that for opposite colourings $j$ and $\bar{j}$ we have $\card{Y_j - Y_{\bar{j}}} \le 2 \sqrt{n \log n}$. We can therefore match the balls $B_d(v_i)$ into isomorphic pairs with opposite colourings, with at most $2\sqrt{n \log n}$ unmatched balls for each of the $\kappa$ isomorphism types. Provided $n$ is suitably large, this is a total of at most $\tfrac12 s'$ unmatched balls, leaving us with at least $\tfrac14 s'$ matched pairs of balls.

\medskip

This shows that the random colouring $\varphi'$ is very likely to satisfy (a), (b) and (c). To finish the proof, we shall show that it is typically also close to being a bisection, and we can make it one without destroying the other properties.

Consider the random variable that is the difference between the numbers of red and blue vertices. Since every path component in $F_1$ of odd length contributes an equal number of red and blue vertices to the colouring, any discrepancy in the colour class sizes must come from the paths of even length. Furthermore, each such path contributes a difference in the colour class sizes of exactly one, with equal probability in either direction. Thus, the expected difference is zero, and recolouring a path component can affect the difference by at most two. We thus make a final appeal to Theorem~\ref{thm:mcdiarmid}, finding that with probability $1 - o(1)$, the difference in the two colour classes is of size at most $\tfrac{1}{20} \sqrt{n \log n}$.

\medskip

In summary, with positive probability it holds that the colouring $\varphi'$ satisfies properties (a), (b) and (c), and that the difference $\Delta$ between the number of red and blue vertices is at most $\tfrac{1}{20} \sqrt{n \log n}$. We can therefore obtain a bisection $\varphi$ by taking $\varphi'$ and replacing the colouring of $\Delta$ even paths in $F_1$ with their opposite colourings, thereby satisfying (d).

Note that $\varphi$ remains a proper colouring of $F_1$, and therefore property (a) is preserved. Moreover, we are changing the colours of at most $5 \Delta$ vertices. As previously discussed, changing the colour of a vertex can affect at most two monochromatic copies of $P_t$, and hence, for each $2 \le t \le 6$, the difference $\card{r_{P_t}(G,\varphi) - b_{P_t}(G,\varphi)}$ deviates from $\card{r_{P_t}(G,\varphi') - b_{P_t}(G, \varphi')}$ by at most $\sqrt{n \log n}$. In particular, we have $\card{r_{P_t}(G,\varphi) - b_{P_t}(G, \varphi)} \le 3 \sqrt{n \log n}$ for each $2 \le t \le 6$, establishing (b).

Similarly, under $\varphi'$ we had at least $\tfrac14 s' \ge \tfrac13 2^{-2d-3} n$ pairs of isomorphic balls with opposite colourings. Since these balls are pairwise disjoint, each recoloured vertex can affect at most one such pair. Hence, under $\varphi$ we still have at least $\tfrac14 s' - 5 \Delta$ such pairs which, if $n$ is large enough, is at least $2^{-2d - 5} n$ pairs. Thus, (c) holds for $\varphi$ as well, completing the proof.
\end{proof}

\subsection{Correcting the bisection} \label{sec:correction}

The colouring $\varphi$ from Proposition~\ref{prop:random} is close to being the bisection we need, but the number of short paths in the red and blue subgraphs can be a little off. In the second stage of our argument, we will make local changes to the colouring to correct these discrepancies. We do this via gadgets we call \emph{$P_t$-reducers}. These are flexible subgraphs, in the sense that they admit two different colourings which have the same monochromatic subgraph counts, except for short paths. By choosing an appropriate colouring, then, we can adjust the values of $r_{P_t}(G,\varphi)$ and $b_{P_t}(G, \varphi)$, making them equal.

When doing so, though, we do need to take care that our changes do not leak out and affect subgraph counts elsewhere. We thus insulate the $P_t$-reducer by colouring its boundary with alternating colours, thus preventing any monochromatic components from extending outwards. 

\begin{definition}[$P_t$-reducer] \label{def:reducer}
Given some $t \ge 3$, an induced subgraph $R \subseteq G$ is a \emph{$P_t$-reducer} if there are two vertex colourings $\psi_1, \psi_2$ of $B_2(R)=R\cup N(R)\cup N^2(R)$ such that:
\begin{itemize}
	\item[(i)] the two colourings have the same number of red (and therefore also blue) vertices,
	\item[(ii)] in both $\psi_1$ and $\psi_2$, all vertices of $N(R)$ are blue and all vertices of $N^2(R)$ are red, 
	\item[(iii)] $r_H(B_2(R), \psi_1) = r_H(B_2(R), \psi_2)$ and $b_H(B_2(R), \psi_1) = b_H(B_2(R), \psi_2)$, unless $H = P_{\ell}$ for some $2 \le \ell \le t$, and
	\item[(iv)] $r_{P_t}(B_2(R), \psi_2) = r_{P_t}(B_2(R), \psi_1) -1 $ and $b_{P_t}(B_2(R),\psi_1) = b_{P_t}(B_2(R), \psi_2)$.
\end{itemize}
\end{definition}

Of course, this definition is only useful if we can actually find $P_t$-reducers in our graph. Fortunately, they happen to be ubiquitous in cubic graphs.

\begin{prop}\label{prop:reducerexistence}
Let $G$ be a connected cubic graph on more than $3 \cdot 2^{50}$ vertices, and let $v \in V(G)$ be arbitrary. Then, for every $3 \le t \le 6$, there is a $P_t$-reducer in $B_{50}(v)$.
\end{prop}

We shall prove Proposition~\ref{prop:reducerexistence} in Section~\ref{sec:reducers}, but first we show how one can use $P_t$-reducers to obtain the desired isomorphic bisection.

\begin{proof}[Proof of Theorem~\ref{thm:main}]
Given a sufficiently large connected cubic graph $G$, set $d = 57$ and let $\varphi_0$ be the bisection given by Proposition~\ref{prop:random} with this $d$. We then have $r_H(G,\varphi_0) = b_H(G,\varphi_0) = 0$ for all $H$ except $H = P_t$, $2 \le t \le 6$. For these paths, we have $\card{r_{P_t}(G,\varphi_0) - b_{P_t}(G,\varphi_0)} \le 3\sqrt{n \log n}$, and we shall correct these imbalances one at a time, in decreasing order of path length.

\medskip

We start with $t = 6$. Suppose, without loss of generality, that we have $r_{P_6}(G,\varphi_0) > b_{P_6}(G,\varphi_0)$. Take the first pair of isomorphic and oppositely-coloured balls, $B_{57}(u_1)$ and $B_{57}(w_1)$. By Proposition~\ref{prop:reducerexistence}, we can find some $P_6$-reducer $R \subseteq G[B_{50}(u_1)]$, and therefore we find a corresponding oppositely-coloured copy $\bar{R} \subseteq G[B_{50}(w_1)]$ as well. We then recolour $B_2(R)$ with $\psi_2$, and colour $B_2(\bar{R})$ with $\bar{\psi}_1$, the opposite colouring of $\psi_1$.

Let $\varphi_1$ be the resulting colouring. Note that this is still a bisection, since $\psi_1$ and $\psi_2$ have the same number of red vertices, and we have made symmetric changes in $B_2(R)$ and $B_2(\bar{R})$. We next claim that $r_H(G,\varphi_1) = b_H(G,\varphi_1)$ for all graphs except $P_t$, $2 \le t \le 6$, and that $r_{P_6}(G,\varphi_1) - b_{P_6}(G,\varphi_1) = r_{P_6}(G,\varphi_0) - b_{P_6}(G, \varphi_0) - 1$; that is, the difference in monochromatic copies of $P_6$ is reduced by one.

Observe that we only need to concern ourselves with monochromatic components contained within $B_{57}(u_1)$ and $B_{57}(w_1)$. Indeed, in $\varphi_0$, all monochromatic components were paths of length at most five. As we only recoloured vertices in $B_2(R)$ and $B_2(\bar{R})$, which are contained in $B_{50}(u_1)$ and $B_{50}(w_1)$ respectively, any affected components do not reach outside the original balls $B_{57}(u_1)$ and $B_{57}(w_1)$. Also note that, since these balls had the opposite colourings, before the recolouring there was a one-to-one correspondence between red components in the first ball and blue components in the second ball and vice versa.

First we consider components not fully contained within $B_2(R)$ or $B_2(\bar{R})$. Recall that in $\psi_2$, the vertices in $N(R)$ and $N^2(R)$ receive opposite colours, with the same being true of the colouring $\bar{\psi}_1$ of $B_2(\bar{R})$. Thus, any such component in $B_{57}(u_1)$ can only contain (red) vertices from $N^2(R)$, together with some vertices in $B_{57}(u_1) \setminus B_2(R)$. However, since $B_{57}(u_1)$ and $B_{57}(w_1)$ are isomorphic and oppositely-coloured, these components are in bijection with isomorphic blue components in $B_{57}(w_1)$, and hence no additional discrepancy is introduced through these components.

This leaves us with components fully within $B_2(R)$ and $B_2(\bar{R})$, where the properties of the $P_6$-reducer come into play. For any component $H$ that is not a path of length at most six, we have $r_H(B_2(R),\psi_2) = r_H(B_2(R),\psi_1) = b_H(B_2(\bar{R}), \bar{\psi}_1)$, 
and similarly $b_H(B_2(R),\psi_2) = r_H(B_2(\bar{R}),\bar{\psi}_1)$. Thus the monochromatic copies of these components remain balanced. As for paths of length six, we have $$r_{P_6}(B_2(R),\psi_2) = r_{P_6}(B_2(R),\psi_1) - 1 = b_{P_6}(B_2(\bar{R}),\bar{\psi}_1) - 1,$$ while $b_{P_6}(B_2(R), \psi_2) = r_{P_6}(B_2(\bar{R}),\bar{\psi}_1)$, and so the difference between red and blue copies of $P_6$ is indeed reduced by one in $\varphi_1$.

\medskip

Although this recolouring could create monochromatic components that are not paths of length at most six, these must be fully contained within the balls $B_{57}(u_1)$ and $B_{57}(w_1)$, which are disjoint from all the other balls. We can thus repeat this process a further $r_{P_6}(G, \varphi_1) - b_{P_6}(G, \varphi_1) - 1$ times, using the next pairs of balls $B_{57}(u_i)$ and $B_{57}(w_i)$ in the sequence. This gives us a sequence of colourings, $\varphi_1, \varphi_2, \hdots, \varphi_k$, the last of which will satisfy $r_{P_6}(G, \varphi_k) = b_{P_6}(G, \varphi_k)$.

Furthermore, since each $P_6$-reducer is of constant size, every step of the process can only have created a constant number of monochromatic $P_5$-components. Thus, we still have $\card{r_{P_5}(G,\varphi_k) - r_{P_5}(G, \varphi_k)} = O \left( \sqrt{n \log n} \right)$. We can therefore use Proposition~\ref{prop:reducerexistence} to find $P_5$-reducers in the next $O \left( \sqrt{n \log n} \right)$ pairs of balls and balance the monochromatic $P_5$ counts. Once those are handled, we proceed to fixing the $P_4$ counts, and then finally the $P_3$ counts. Observe that we require a total of $O \left( \sqrt{n \log n} \right)$ steps, and since Proposition~\ref{prop:random} guarantees us $\Omega(n)$ pairs of isomorphic balls, we can see this process through to completion.

Let $\varphi$ be the final colouring obtained. Following our corrections with the $P_t$-reducers, we know that for every $H \neq P_2$, we have $r_H(G,\varphi) = b_H(G, \varphi)$. However, since $G$ is cubic and $\varphi$ is a bisection, double-counting the edges between the colour classes shows that the total number of monochromatic red edges must equal the number of blue edges. It therefore follows that $r_{P_2}(G,\varphi) = b_{P_2}(G, \varphi)$ as well, and hence the subgraphs induced by the red and the blue vertices are isomorphic.
\end{proof}

\section{Constructing $P_t$-reducers} \label{sec:reducers}

To complete the proof of Theorem~\ref{thm:main}, we need to prove Proposition~\ref{prop:reducerexistence}, showing that we can find $P_t$-reducers in the local neighbourhoods of every vertex. We will first show a very simple construction that works in graphs of girth at least seven, thereby providing a short proof of Conjecture~\ref{conj:ando} for large cubic graphs without short cycles. The proof in the general case is a little more involved, requiring analysis of a few different cases, and can be found in Section~\ref{sec:general}.

\subsection{Graphs of girth at least seven} \label{sec:largegirth}

Our construction is based around geodesics --- that is, shortest paths in the graph between a pair of vertices. We start with a lemma about neighbourhoods of geodesics.

\begin{lemma} \label{lem:geonhd}
Let $G$ be a cubic graph of girth at least seven, and let $P$ be a geodesic in $G$. Then $N(P)$, the set of external neighbours of vertices on $P$, is an independent set.
\end{lemma}

\begin{proof}
Let $u, w \in N(P)$ be arbitrary vertices in the neighbourhood of $P$, and suppose for contradiction we have an edge $\{u,w\}$. Let $u', w' \in P$ be neighbours of $u$ and $w$ respectively. If $u'$ and $w'$ are at distance at most three along the path $P$, then, together with the edges $\{u', u\}, \{u,w\}$ and $\{w,w'\}$, we would obtain a cycle of length at most six, contradicting our assumption on the girth of $G$.

On the other hand, if $u'$ and $w'$ are at distance at least four on the path $P$, then we could shorten the path by rerouting it between $u'$ and $w'$ through $u$ and $w$ instead. This contradicts $P$ being a geodesic.

Hence $u$ and $w$ cannot be adjacent, and thus it follows that $N(P)$ is an independent set.
\end{proof}

This shows that the structure around a geodesic is particularly simple. As a result, we can easily find $P_t$-reducers, as we now show.

\begin{proof}[Proof of Proposition~\ref{prop:reducerexistence} (girth $\ge 7$)]
Since there are at most $3 \cdot 2^{t+1}<3 \cdot 2^{50}$ vertices at distance at most $t$ from the vertex $v$, there must be a vertex $w$ at distance exactly $t+1$ from $v$. Let $P$ be a geodesic from $v$ to $w$, with vertices $v = v_0, v_1, v_2, \hdots, v_t, v_{t+1} = w$.  We will show that $P$ is a $P_t$-reducer.

We now define the colourings $\psi_1$ and $\psi_2$ of $B_2(P)=P\cup N(P)\cup N^2(P)$. In both colourings, we colour all vertices in $N^2(P)$ red and all vertices in $N(P)$ blue. When colouring the path $P$, we make all vertices red, except in $\psi_1$ the vertex $v_t$ is coloured blue, while in $\psi_2$ the vertex $v_{t-1}$ is coloured blue instead.
\begin{figure}[h]
  \centering
    \includegraphics[width=0.7\textwidth]{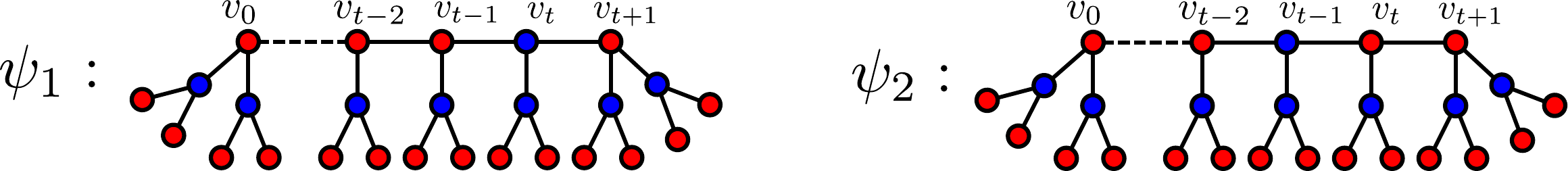}
\end{figure}

Clearly both $\psi_1$ and $\psi_2$ have the same number of red vertices, and thus property (i) is satisfied.  Property (ii) is satisfied by the definition of the colouring on $N(P)$ and $N^2(P)$.

Finally, we inspect the monochromatic components of $B_2(P)$ under $\psi_1$ and $\psi_2$ to verify properties (iii) and (iv). Let us start with the blue components. In the two colourings, the blue vertices are those in $N(P)$, together with one of $v_{t-1}$ or $v_t$. By Lemma~\ref{lem:geonhd}, the vertices in $N(P)$ form an independent set. As $v_{t-1}$ and $v_t$ are internal vertices on the path $P$, they have exactly one neighbour in $N(P)$. Thus in both $\psi_1$ and $\psi_2$, the blue vertices induce one isolated edge and $\card{N(P)} - 1$ isolated vertices, and thus $b_H(B_2(P),\psi_1) = b_H(B_2(P), \psi_2)$ for every $H$.

Since $N(P)$ is entirely blue, the red components of $B_2(P)$ are either wholly contained in $N^2(P)$ or in $P$. In the former case, since all vertices in $N^2(P)$ are red, both colourings have the same components. In the latter case, the path $P$ is broken into two red paths by the sole blue vertex, which is $v_t$ in $\psi_1$ and $v_{t-1}$ in $\psi_2$. Thus $\psi_1$ has a red $P_t$ and a red $P_1$, while $\psi_2$ has a red $P_{t-1}$ and a red $P_2$. This shows that (iii) and (iv) are indeed satisfied, and so the colourings $\psi_1$ and $\psi_2$ bear witness to $P$ being a $P_t$-reducer.
\end{proof}

\subsection{The general case} \label{sec:general}

In general, we cannot expect such simple structure around geodesic paths. However, induced paths will still play a key role in our construction of $P_t$-reducers, as shown by the following lemma.

\begin{lemma}\label{Lemma_2_edges_on_path}
In a cubic graph $G$, let $Q$ be an induced path of length $t+1$ or $t+2$ with vertices labelled as $Q=(u, x,z, q_{1}, \dots, q_{t-1})$ or $Q=(u, x,y, z, q_{1}, \dots, q_{t-1})$ respectively. Let $v$ be a vertex outside $Q$ with the edges $vx$ and $vz$ present. Then $B_{3}(Q)$ contains a $P_t$-reducer.
\end{lemma}
\begin{proof}
Depending on the length of $Q$, define $\psi_1$ and $\psi_2$ on $R:=Q\cup \{v\}$ as follows: 
\begin{figure}[h]
  \centering
    \includegraphics[width=0.7\textwidth]{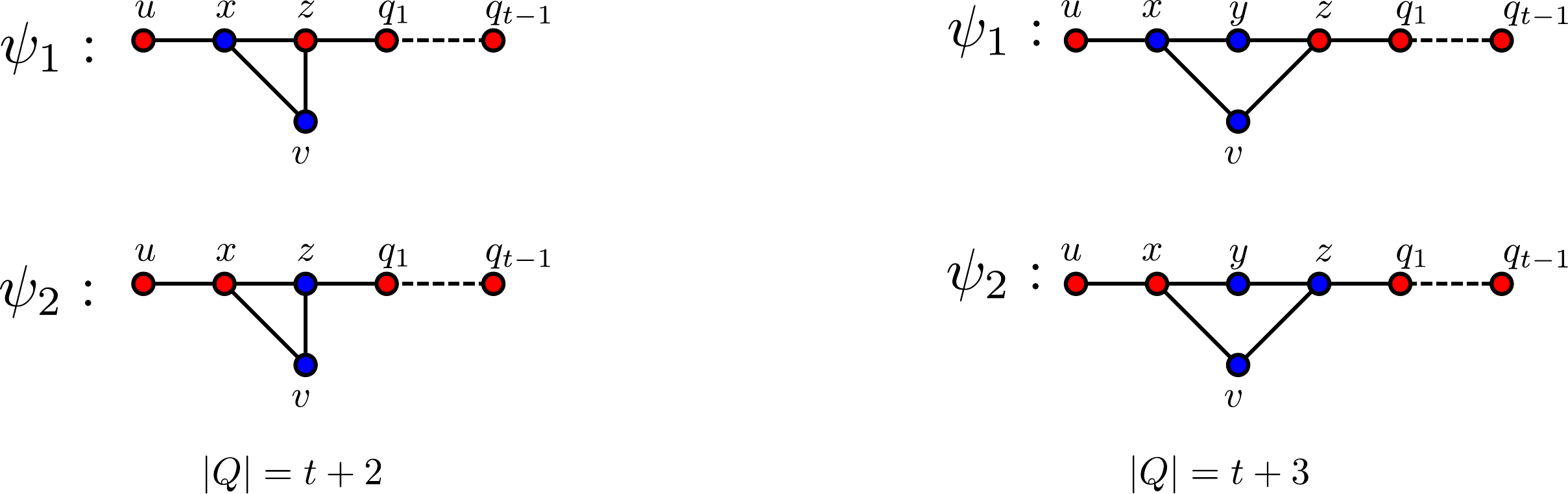}
\label{Figure_induced_path_lemma}
\end{figure}

Define both $\psi_1$ and $\psi_2$ to be blue on $N(R)$ and red on  $N^2(R)$. We claim that this makes $R$ a $P_t$-reducer. Properties (i) and (ii) are immediate by construction. Notice that $b_{H}(B_2(R),\psi_1)=b_{H}(B_2(R),\psi_2)$ for all $H$. Indeed the only blue component that changes between the two colourings is the one that contains $v$. However, since the graph is cubic, neither $x$ nor $z$ has any neighbours in $N(R)$, and therefore
it is easy to see that this component is isomorphic between the two colourings. 
Finally notice that the only red components that change between the two colourings are the two red subpaths of $Q$. 
The colouring $\psi_1$ has a red $P_t$ and a red $P_1$, while $\psi_2$ has a red $P_{t-1}$ and a red $P_2$.
We also have $r_{H}(B_2(R),\psi_1)=r_{H}(B_2(R),\psi_2)$ for all other $H$. This completes the proof of (iii) and (iv).
\end{proof}
Recall that a geodesic path $Q$ is a shortest length path between its endpoints, and note that a length $14$ geodesic path exists in $B_{14}(v)$ for every vertex $v$ in a large connected cubic graph.
An important consequence of the above lemma is that it reduces the problem to the case when every vertex outside a geodesic path $Q$ has at most one neighbour on $Q$. Using this we can prove the following lemma.  
\begin{lemma}\label{Lemma_unbalanced_reducer}
Let $Q$ be a geodesic path of length $14$ in a cubic graph $G$. Then $B_3(Q)$ either contains $P_t$-reducers for all  $3\leq t\leq 6$ or it contains, for each $3 \le t \le 6$, an induced subgraph $R$ which has 
two colourings $\psi_1$ and $\psi_2$ of $B_2(R)$ that satisfy (ii), (iii) and (iv) of the definition of $P_{t}$-reducer and also:
\begin{itemize}
    \item [(i')] $\psi_1$ has one more red vertex than $\psi_2$.
\end{itemize}
\end{lemma}
\begin{proof}
Let $Q=(q_1, \dots, q_{15})$.
For each $2 \le i \le 14$, let $r_i$ be the   neighbour of $q_i$ outside $Q$. These are unique because the graph is cubic. If $r_i=r_j$ for some $2 \le i< j \le 14$, the fact that $Q$ is a geodesic path forces $j=i+ 1$ or $j=i+ 2$. In either case, we find $P_t$-reducers for all $3\leq t\leq 6$ using Lemma~\ref{Lemma_2_edges_on_path}. Thus we can assume that $r_2, \dots, r_{14}$ are all distinct.

Next, suppose that $r_ir_{i+1}$ is an edge for  all $3 \leq i\leq 10$. Since $G$ is cubic, this determines all the neighbours of $r_i$ for $4 \le i \le 10$. It follows that for $t\in \{3,4,5,6\}$, the path $Q'=(q_2, q_3, q_4, r_4, r_5, \dots, r_{t+4})$ is induced, has length $t+3$, and the vertex $v:=r_3$ has neighbours $q_3$ and $r_4$. Thus Lemma~\ref{Lemma_2_edges_on_path} applies to give us $P_t$-reducers for all  $3 \le t \le 6$.

Hence, we may assume that $r_ir_{i+1}$ is a non-edge for  some $3\leq i\leq 10$. Fix $t\in \{3,4,5,6\}$.  Define the subpath $Q'=(q_{i-1}, \dots, q_{i+t-1})$
Define $\psi_1$ and $\psi_2$ on $R:=Q'\cup \{r_i, r_{i+1}\}$ as follows: 
\begin{figure}[h]
  \centering
    \includegraphics[width=0.7\textwidth]{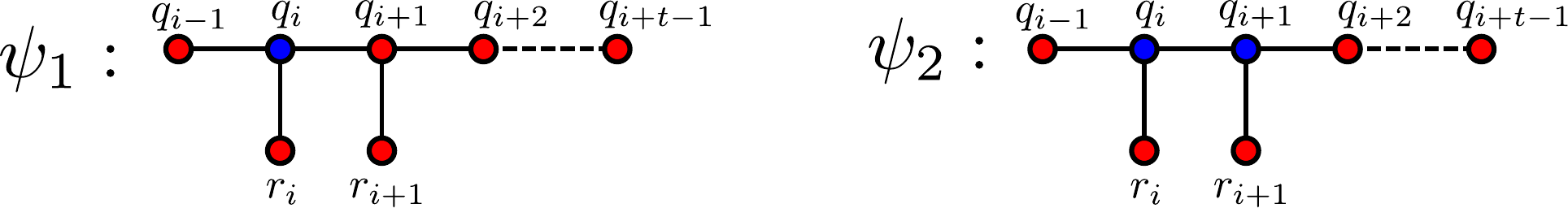}
\end{figure}

Extend both $\psi_1$, and $\psi_2$ to be blue on $N(R)$ and red on  $N^2(R)$. We claim that then these colourings of $B_2(R)$ satisfy (i'), (ii), (iii), and (iv). 
Properties (i') and (ii) are immediate by construction. For (iii) and (iv), notice that the only components that change between the two colourings are the blue component containing $q_i$ (which, by virtue of the fact that all neighbours of $q_i$ and $q_{i+1}$ lie within $R$, is a $P_1$ in $\psi_1$ and a $P_2$ in $\psi_2$) and the red component containing $q_{i+2}$ (which is a $P_t$ in $\psi_1$ and splits into a $P_1$ and a $P_{t-2}$ in $\psi_2$).
\end{proof}

The above lemma implies  Proposition~\ref{prop:reducerexistence} for $t\geq 4$.
\begin{proof}[Proof of Proposition~\ref{prop:reducerexistence} for $t\geq 4$]
Note that within $B_{45}(v)$, we can find two geodesic paths of length $14$, which we call $Q$ and $R$, at distance $10$ from each other. Thus the balls $B_3(Q)$ and $B_3(R)$ are disjoint. If either of these balls contains a $P_t$-reducer, we are done. Otherwise, by Lemma~\ref{Lemma_unbalanced_reducer}, we can assume the existence of some $S_1 \subseteq B_3(Q)$ with colourings $\psi_1^{S_1,t}$ and $\psi_2^{S_1,t}$ of $B_2(S_1)$ satisfying (ii), (iii), (iv) of the definition of a $P_t$-reducer and also (i') from Lemma~\ref{Lemma_unbalanced_reducer}. Similarly, we find $S_2 \subseteq B_3(R)$ with colourings $\psi_1^{S_2,t-1}, \psi_2^{S_2,t-1}$ of $B_2(S_2)$ satisfying (ii), (iii), (iv) of the definition of a $P_{t-1}$-reducer and also (i') above. We now construct colourings $\psi_1$ and $\psi_2$ of $B_{50}(v)$ satisfying the definition of a $P_t$-reducer.
\begin{itemize}
    \item On $B_2(S_1)$, $\psi_1$ and $\psi_2$ agree with $\psi_1^{S_1,t}$ and $\psi_2^{S_1,t}$ respectively.
    \item  On $B_2(S_2)$, $\psi_1$ and $\psi_2$ have the opposite colourings of $\psi_1^{S_2,t-1}$ and $\psi_2^{S_2,t-1}$.
    \item Outside $B_2(S_1)$ and  $B_2(S_2)$,  $\psi_1$ and $\psi_2$ are entirely blue, except on $N^{50}(v)$, where they are red.
\end{itemize}
We check that these colourings do indeed result in a $P_t$-reducer $B_{48}(v)$. Using property (i') of the pairs $\psi_1^{S_1,t}, \psi_2^{S_1,t}$ and $\psi_1^{S_2,t-1}, \psi_2^{S_2,t-1}$, and the fact that we use the opposite colourings on $S_2$, we see that $\psi_1$ and $\psi_2$ have an equal number of red vertices in total, satisfying property (i) of a $P_t$-reducer. Property (ii) holds by the third bullet point of the construction. Property (iii) is immediate from it holding for the colourings $\psi_1^{S_1,t}$, $\psi_2^{S_1,t}$, $\psi_1^{S_2,t-1}$ and $\psi_2^{S_2,t-1}$, since the property (ii) for these colourings ensures that the monochromatic components involving vertices whose colours change are fully contained within $B_2(S_1)$ and $B_2(S_2)$ respectively. For property (iv), note we lose exactly one red copy of $P_t$ when going from $\psi_1^{S_1,t}$ to $\psi_2^{S_1,t}$ on $B_2(S_1)$, while the number of blue copies of $P_t$ stays the same. On the other hand, going from $\psi_1^{S_2,t-1}$ to $\psi_2^{S_2,t-1}$ on $B_2(S_2)$ does not affect the number of the monochromatic copies of $P_t$ (by property (iii)). Thus in total exactly one red $P_t$ is lost, as required. \end{proof}

The above proof doesn't work for $t=3$ because our proof of Lemma~\ref{Lemma_unbalanced_reducer} doesn't work for $t=2$. Thus for  $t=3$, we need a different proof of the proposition.
\begin{proof}[Proof of Proposition~\ref{prop:reducerexistence} for $t=3$]
Let $Q=(q_0, \dots, q_{20})$ be a length $20$ geodesic within distance $20$ of $v$. For each $1 \le i \le 19$, let $r_i$ be the (unique) neighbour of $q_i$. As in the proof of Lemma~\ref{Lemma_unbalanced_reducer}, if $r_i=r_j$ for some $1 \leq i< j\leq 19$, we must have $j=i+ 1$ or $j=i+ 2$, and then we find a $P_3$-reducer using Lemma~\ref{Lemma_2_edges_on_path}. Thus we may assume that $r_1, \dots, r_{19}$ are all distinct.

\begin{claim}\label{Claim34}
We have at least one of the following (see Figure~\ref{FigureClaim34}):
\begin{enumerate}[(a)]
    \item For some $3 \le i \le 9$, none of the edges $r_ir_{i+1}, r_ir_{i+2}$ or $r_{i+1}r_{i+2}$ are present.
    \item For some $3 \le i \le 10$, the edge $r_{i}r_{i+1}$ is present and either \begin{itemize}
        \item  the edges $r_ir_{i+2}$ and $r_ir_{i+3}$ are absent, or
        \item the edges $r_{i-1}r_{i+1}$ and $r_{i-2}r_{i+1}$ are absent.
    \end{itemize} 
    \item For some $3 \le i \le 8$, the edge $r_{i}r_{i+1}$ is present, the edges $r_{i+1}r_{i+2}, r_{i+1}r_{i+3}, r_{i-1}r_i$ and $r_{i-2}r_i$ are absent, and either $r_ir_{i+3}$ or $r_{i-2}r_{i+1}$ is present. 
\end{enumerate}
\end{claim}
\begin{figure}[h]
  \centering
    \includegraphics[width=0.6\textwidth]{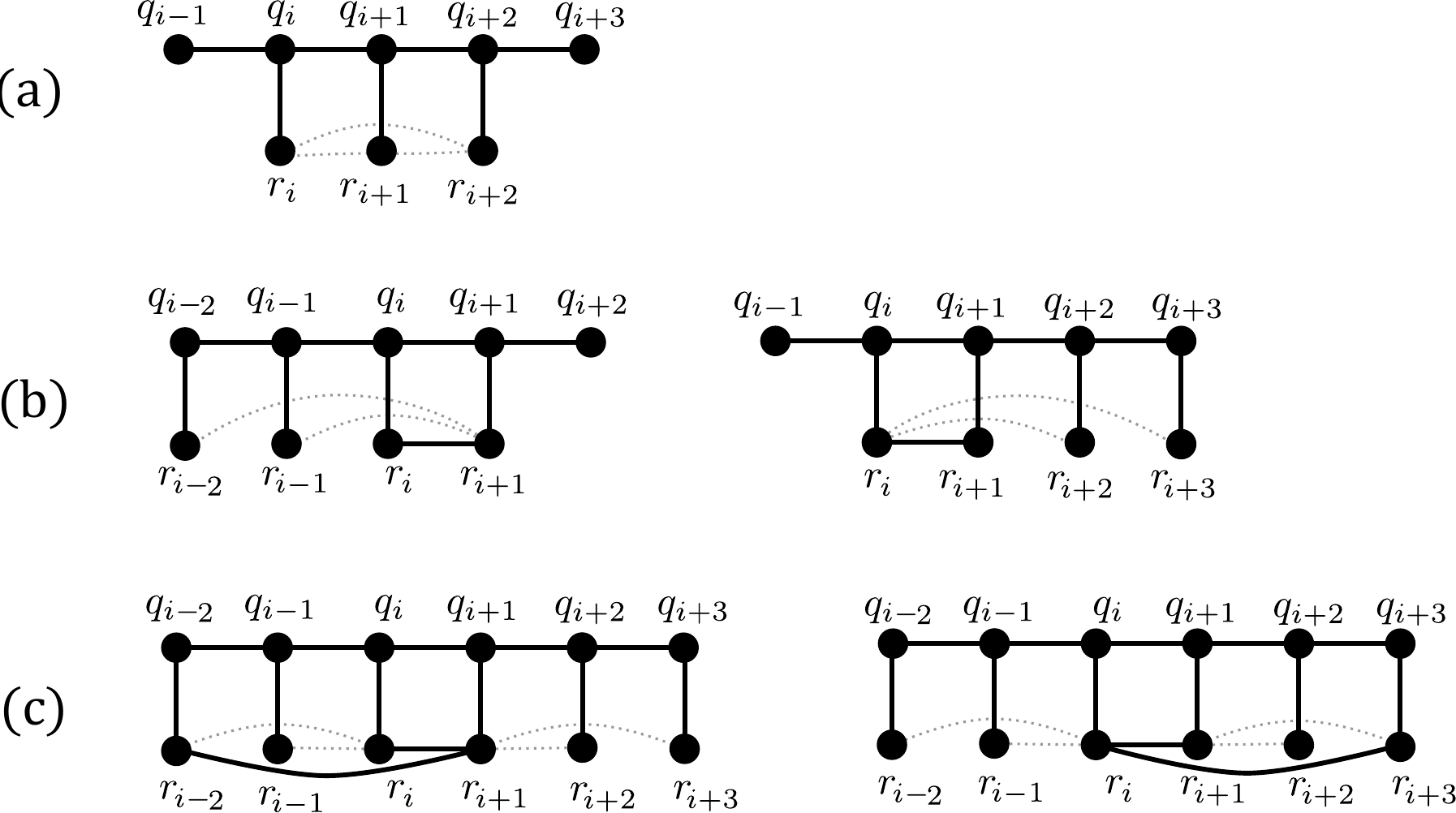}
  \caption{The three cases for Claim~\ref{Claim34}. Solid lines represent edges that are present while dashed lines represent edges which are absent (whereas unmarked edges may or may not be present).}\label{FigureClaim34}
\end{figure}

\begin{proof}
First observe that the three edges $r_3 r_5, r_5 r_7$ and $r_7 r_9$ cannot all be present, as that would contradict $Q$ being a geodesic. Hence, if (a) does not hold, we must have the edge $r_i r_{i+1}$ for some $3 \le i \le 8$.

Next, suppose that $r_{i+1}r_{i+2}$ is also present. In this case, $N(r_{i+1})=\{q_{i+1}, r_i, r_{i+2}\}$, which implies that $r_{i-1}r_{i+1}$ and $r_{i-2}r_{i+1}$ are both absent (since $r_{i-1}, r_{i-2}$ are distinct from $q_{i+1}, r_i, r_{i+1}$). This leaves us in case (b). The same argument shows that we are in case (b) if any of the edges   $r_{i+1}r_{i+3}, r_{i-1}r_i$ or $r_{i-2}r_i$ were present. Thus we can assume that $r_{i+1}r_{i+2}, r_{i+1}r_{i+3}, r_{i-1}r_i$ and $r_{i-2}r_i$ are absent.

If (c) doesn't hold, then both $r_ir_{i+3}$ and $r_{i-2}r_{i+1}$ must be absent. Then, to avoid case (b), both $r_ir_{i+2}$ and $r_{i-1}r_{i+1}$ must be present. Hence $N(r_i)=\{q_i, r_{i+1}, r_{i+2}\}$ and $N(r_{i+1})=\{q_{i+1}, r_{i-1}, r_i\}$.
Recalling that  $r_{i+1}r_{i+2}$ and $r_{i+1}r_{i+3}$ are absent, note that $r_{i+2}r_{i+3}$ must be present, as otherwise we would be in case (a) for $i' = i+1$. However, as $r_ir_{i+3}$ and $r_{i+1}r_{i+3}$ are both absent, we are then in case (b) with $i'=i+2$.
\end{proof}
We now find $P_3$-reducers in each of the above three cases. 
\begin{enumerate}[(a)]
    \item Set $R=\{q_{i-1}, q_i, q_{i+1}, q_{i+2}, q_{i+3}, r_i, r_{i+1}, r_{i+2} \}$ and define colourings $\psi_1, \psi_2$ on $R$ as follows: \begin{figure}[h]
  \centering
    \includegraphics[width=0.7\textwidth]{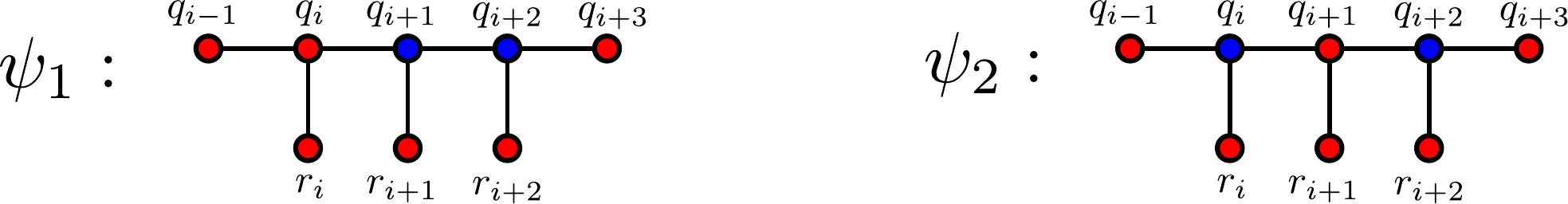}
\end{figure}

Define both $\psi_1$, and $\psi_2$ to be blue on $N(R)$ and red on  $N^2(R)$. It is immediate that this colouring satisfies the definition of $P_3$-reducer. 

    \item Without loss of generality, we may suppose that $r_ir_{i+1}$ is present and the edges $r_ir_{i+2}, r_ir_{i+3}$ are absent (the other case is symmetric to this by reversing the order of the path $Q$). Let $u$ be the third neighbour of $r_i$, i.e. $N(r_{i})=\{q_i, r_{i+1}, u\}.$
    Consider the path $Q'=(u, r_i, q_i, q_{i+1}, q_{i+2}, q_{i+3})$. We claim that this is an induced path. Indeed the edges $r_iq_{j}$ are all absent for $j\neq i$, so the only way this could be non-induced is if $u=r_{i+1},  r_{i+2}$ or $r_{i+3}$. But $u$ was the third neighbour of $r_i$ so $u \neq r_{i+1}$ by definition, and $r_{i+2}, r_{i+3}\not\in N(r_i)$ by the definition of case (b). Thus $Q'$ is indeed induced, and by Lemma~\ref{Lemma_2_edges_on_path} (applied with $v=r_{i+1}$) there is a $P_3$-reducer in $B_3(Q')$.
    
    \item 
     Without loss of generality, we may suppose that $r_ir_{i+1}$ and $r_{i}r_{i+3}$ are present  and $r_{i+1}r_{i+2},$ $r_{i+1}r_{i+3},$ $r_{i-1}r_i,$ $r_{i-2}r_i$ are absent (the other case is symmetric).
     Set $R=\{q_{i-1}, q_i, q_{i+1}, q_{i+2}, q_{i+3}, r_i, r_{i+1}, r_{i+2}, r_{i+3} \}$ and define colourings $\psi_1, \psi_2$ on $R$ as shown below. Extend both $\psi_1$ and $\psi_2$ to be blue on $N(R)$ and red on  $N^2(R)$. Since $q_i, r_i, q_{i+1}$ have no neighbours in $N(R)$, it is easy to check that this colouring satisfies the definition of $P_3$-reducer.
\end{enumerate}

\begin{figure}[h!]
  \centering
    \includegraphics[width=0.7\textwidth]{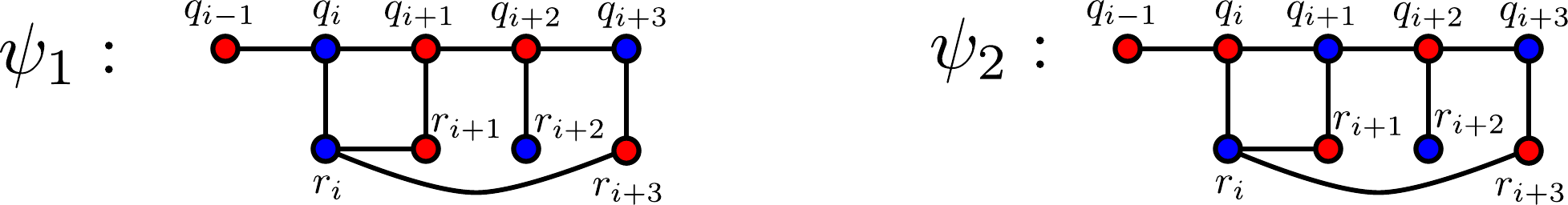}
\end{figure}

\end{proof}

\section{Concluding remarks} \label{sec:conclusion}

In this paper we have proven Ando's conjecture for all large connected cubic graphs. While this only leaves a finite number of graphs to be checked, and there is room to optimise the constants in our proof, there will still be too many cases to be handled computationally. Hence, a complete resolution of Conjecture~\ref{conj:ando} will likely require some additional ideas. In this section we indicate some ways in which our proof could be modified, which might help make progress towards the full conjecture, and close with some related open problems.

\paragraph{A simpler starting block}

While our proof is relatively short, one could argue that it is not fully self-contained, as we use Thomassen's theorem, which is already a very significant result. However, it is not crucial in our proof that, in the partition of the edges of $G$ into linear forests $F_1$ and $F_2$, the paths have length at most five. We could therefore replace Theorem~\ref{thm:thomassen} with one of its predecessors (the theorems of Aldred and Wormald~\cite{AW98} or Jackson and Wormald~\cite{JW96}), which have simpler proofs, but allow for longer paths. Although this comes at the cost of requiring $P_t$-reducers for larger values of $t$, our constructions readily generalise to longer paths. This is especially easy when we assume $G$ has large girth, resulting in a truly short proof of this special case.

\paragraph{Fewer reducers}

Alternatively, one might seek to reduce the amount of work done in Section~\ref{sec:reducers}, when constructing the $P_t$-reducers. A potential route to simplification lies in the observation that, when using Theorem~\ref{thm:thomassen}, it was not very important that the paths in $F_1$ were so strongly bounded in length. Indeed, we only used the lengths of the paths in $F_1$ to bound the Lipschitz constant $c$ in our application of Theorem~\ref{thm:mcdiarmid}, and we can afford for this to be as large as $n^{o(1)}$. On the other hand, if we can limit the lengths of the paths in $F_2$ to some $\ell \le 4$, then we would only need to construct $P_t$-reducers for $2 \le t \le \ell+1$.

\begin{ques} \label{ques:asymmetric}
What is the smallest $\ell$ for which the edges of any connected cubic graph on $n$ vertices can be partitioned into two spanning linear forests $F_1$ and $F_2$, such that the paths in $F_1$ are of length $n^{o(1)}$, and the paths in $F_2$ are of length at most $\ell$?
\end{ques}

It is worth noting that, while the five in Theorem~\ref{thm:thomassen} is best possible, the examples of tightness given by Thomassen~\cite{Tho99} are the two cubic graphs on six vertices. It would be interesting to know if there are arbitrarily large tight constructions, or if, when dealing with large connected graphs, one can achieve $\ell = 4$ even in the symmetric setting. Furthermore, we can weaken Question~\ref{ques:asymmetric}, as in our application $F_1$ does not have to be a linear forest, but rather a bipartite graph with bounded components.

\paragraph{Stronger conjectures}

While Theorem~\ref{thm:main} sheds light on the structure of large cubic graphs, showing that they can be partitioned into isomorphic induced subgraphs, it does not directly address the motivating question raised in Section~\ref{sec:intro}, as there are no guarantees that these subgraphs are simple. However, by analysing our proof, one can obtain some further information about the subgraphs obtained. As stated in Proposition~\ref{prop:random}, the only monochromatic components in the initial random colouring are paths of length at most five. When we then use the $P_t$-reducers to make the subgraphs isomorphic, we can introduce more complicated monochromatic components. However, since the $P_t$-reducers are all isolated within balls of bounded radius, it follows that the components in the isomorphic subgraphs are of bounded size.

In particular, if we assume that our connected graph has large girth,\footnote{By carefully considering the $P_t$-reducers from Section~\ref{sec:largegirth}, it suffices to assume girth at least $15$.} it follows that the isomorphic subgraphs are forests. Moreover, we need never have vertices of degree three in the isomorphic subgraphs, as these can be recoloured (in pairs) to become isolated vertices of the opposite colour. Thus, we in fact partition large connected cubic graphs of large girth into isomorphic linear forests. It was conjectured by Abreu, Goedgebeur, Labbate and Mazzuoccolos~\cite{AGLM19} that every cubic graph should admit such a partition; the challenge lies in removing the girth condition.

Ban and Linial~\cite{BL16} went even further, conjecturing that much more should be true when we restrict our attention to two-edge-connected cubic graphs.

\begin{conj} \label{conj:banlinial}
The vertices of every bridgeless cubic graph, with the exception of the Petersen graph, can be two-coloured such that the two colour classes induce isomorphic matchings.
\end{conj}

\noindent The conjecture remains open, although it has been proven for three-edge-colourable graphs by Ban and Linial~\cite{BL16} and for claw-free graphs by Abreu, Goedgebeur, Labbate and Mazzuoccolo~\cite{AGLM18}, while Esperet, Mazzuoccolo and Tarsi~\cite{EMT17} proved that any counterexample must have circular flow number equal to five. It would be very interesting to see to what extent our methods can be applied to this conjecture, as well as to Wormald's conjecture on partitioning the edges of cubic graphs into isomorphic linear forests.

\subsection*{Acknowledgement} We are grateful to the anonymous reviewers for their valuable comments and helpful suggestions.


\begin{thebibliography}{99}

\bibitem{AGLM18}
M.~Abreu, J.~Goedgebeur, D.~Labbate and G.~Mazzuoccolo,
\newblock{A note on $2$-bisections of claw-free cubic graphs},
\newblock{\em Discrete Appl. Math.} {\bf 244} (2018), 214--217.

\bibitem{AGLM19}
M.~Abreu, J.~Goedgebeur, D.~Labbate and G.~Mazzuoccolo,
\newblock{Colourings of cubic graphs inducing isomorphic monochromatic subgraphs},
\newblock{\em J. Graph Theory} {\bf 92} (2019), 415--444.

\bibitem{AC81}
J.~Akiyama and V.~Chv\'atal,
\newblock{A short proof of the linear arboricity for cubic graphs},
\newblock{\em Bull. Liber. Arts Sci. NMS} {\bf 2} (1981).

\bibitem{AEH80}
J.~Akiyama, G.~Exoo and F.~Harary,
\newblock{Covering and packing in graphs. II. Cyclic and acyclic invariants},
\newblock{\em Math. Slovaca} {\bf 30} (1980), 405--417.

\bibitem{AW98}
R.~E.~L.~Aldred and N.~C.~Wormald,
\newblock{More on the linear $k$-arboricity of regular graphs},
\newblock{\em Australas. J. Combin.} {\bf 18} (1998), 97--104.

\bibitem{ADOV03}
N.~Alon, G.~Ding, B.~Oporowski and D.~Vertigan,
\newblock{Partitioning into graphs with only small components},
\newblock{\em J. Combin. Theory Ser. B} {\bf 87} (2003), 231--243.

\bibitem{BL16}
A.~Ban and N.~Linial,
\newblock{Internal partitions of regular graphs},
\newblock{\em J. Graph Theory} {\bf 83} (2016), 5--18.

\bibitem{BFHP84}
J.~C.~Bermond, J.~L.~Fouquet, M.~Habib and B.~P\'eroche,
\newblock{On linear $k$-arboricity},
\newblock{\em Discrete Math.} {\bf 52} (1984), 123--132.

\bibitem{EMT17}
L.~Esperet, G.~Mazzuoccolo and M.~Tarsi,
\newblock{Flows and bisections in cubic graphs},
\newblock{\em J. Graph Theory} {\bf 86} (2017), 149--158.

\bibitem{FTVW09}
J.L. Fouquet, H. Thuillier, J.M. Vanherpe,  and A.P. Wojda,
\newblock{On isomorphic linear partitions in cubic graphs},
\newblock{\em Discrete Math.} {\bf 309} (2009), 6425--6433.
 
\bibitem{HST03}
P.~Haxell, T.~Szab\'o and G.~Tardos,
\newblock{Bounded size components --- partitions and transversals},
\newblock{\em J. Combin. Theory Ser. B} {\bf 88} (2003), 281--297.

\bibitem{JW96}
B.~Jackson and N.~C.~Wormald,
\newblock{On the linear $k$-arboricity of cubic graphs},
\newblock{\em Discrete Math.} {\bf 162} (1996), 293--297.

\bibitem{LMST}
N. Linial, J. Matou\v sek, O. Sheffet and G. Tardos,
\newblock{Graph coloring with no large monochromatic components},
\newblock{\em Comb. Probab. Comput.} {\bf 17} (2008),  577--589.

\bibitem{Lov66}
L.~Lov\'asz,
\newblock{On decomposition of graphs},
\newblock{\em Studia Sci. Math. Hungarica} {\bf 1} (1966), 237--238.

\bibitem{McD89}
C.~McDiarmid,
\newblock{On the method of bounded differences},
\newblock{\em Surveys in Combinatorics}, ed. J.~Siemons, London Mathematical Society Lecture Note Series {\bf 141} (1989), 148--188.

\bibitem{Tho99}
C.~Thomassen,
\newblock{Two-coloring the edges of a cubic graph such that each monochromatic component is a path of length at most $5$},
\newblock{\em J. Comb. Theory Ser. B} {\bf 75} (1999), 100--109.

\bibitem{Woo18}
D.~Wood,
\newblock{Defective and clustered graph colouring},
\newblock{\em Electron. J. Combin.} (2018), \#DS23.

\bibitem{Wor87}
N.~C.~Wormald,
\newblock{Problem 13},
\newblock{\em Ars Combinatorica} {\bf 23} (1987), 332--334.

\end{thebibliography}
\end{document}